\documentclass[11pt]{article}

\usepackage[T1]{fontenc}
\usepackage[utf8]{inputenc}
\usepackage{amsmath,amssymb,amsthm}
\usepackage{geometry}
\usepackage{microtype}
\usepackage[hidelinks]{hyperref}

\newtheorem{theorem}{Theorem}[section]
\newtheorem{lemma}[theorem]{Lemma}
\newtheorem{proposition}[theorem]{Proposition}
\newtheorem{corollary}[theorem]{Corollary}
\newtheorem{conjecture}[theorem]{Conjecture}
\theoremstyle{definition}
\newtheorem{definition}[theorem]{Definition}
\newtheorem*{maintheorem}{Main Theorem}
\numberwithin{equation}{section}

\title{Prime and Nonprime Totatives: A Sharp Construction and Exact Thresholds}
\author{
Amirali Fatehizadeh\\
\small Faculty of Mathematical Sciences, Shahid Beheshti University, Tehran, Iran\\
\small \texttt{A.fatehizadeh@mail.sbu.ac.ir}\\
\small \texttt{Amirali.fatehizadeh@gmail.com}
}
\date{}

\begin{document}

\maketitle

\begin{abstract}
For \(n\ge2\), let \(A(n)=\pi(n)-\omega(n)\) and \(B(n)=\phi(n)-\pi(n)+\omega(n)\) denote the numbers of prime and nonprime totatives of \(n\), respectively, with \(1\) included in \(B(n)\). We construct an explicit injectively parametrized family \(\mathcal C_n\) of composite totatives by completing suitably restricted squarefree products of prime totatives with a larger prime totative that is the unique largest prime factor, making the parametrization recoverable. We prove \(\lvert\mathcal C_n\rvert/A(n)\ge (e^{-\gamma}-o(1))\log A(n)/\log\log A(n)\), realizing the sharp classical leading constant \(e^{-\gamma}\). For the same family, with the same defining parameters, we obtain an effective refinement with error of order \((\log\log A(n))^{-1/4}\) and absolute effectively computable constants. Classical minimal-order results and the prime number theorem show that \(e^{-\gamma}\) is optimal: no family contained in the nonprime totatives can satisfy a uniform lower bound at this scale with a larger leading constant. We also determine exact eventual thresholds. Let \(N_k\) and \(M_k\) be the least integers such that \(\phi(n)>k\pi(n)\) for every \(n\ge N_k\) and \(B(n)>kA(n)\) for every \(n\ge M_k\), respectively. We prove \(M_k\le N_{k+1}\) for every \(k\ge1\), determine \(N_1,\ldots,N_6\) and \(M_1,\ldots,M_5\) exactly, and obtain \(M_k=N_{k+1}\) for \(1\le k\le5\).
\end{abstract}



\section{Introduction}\label{sec:introduction}

For \(n\ge2\), let \(\omega(n)\) denote the number of distinct prime divisors of \(n\), and define
\[
\begin{aligned}
A(n)&=\#\{p\le n:p\text{ is prime and }p\nmid n\}=\pi(n)-\omega(n),\\
B(n)&=\#\{m\le n:(m,n)=1,\ m\text{ is not prime}\}=\phi(n)-\pi(n)+\omega(n).
\end{aligned}
\]
Thus \(A(n)\) counts the prime totatives of \(n\), while \(B(n)\) counts the nonprime totatives, including \(1\). In particular, \(\phi(n)=A(n)+B(n)\) and \(\pi(n)=A(n)+\omega(n)\). These decompositions turn the comparison between Euler's totient function and the prime-counting function into a more concrete counting problem: how many nonprime totatives can be produced from the prime totatives themselves?

Direct multiplicative arguments for this problem have a substantial history. Moser's study of the equation \(\phi(n)=\pi(n)\) already separates, in different notation, the distinct prime divisors, prime totatives, and nonprime totatives, and uses products of prime totatives in establishing the eventual inequality \(\phi(n)>\pi(n)\) \cite[pp.~178--180]{Moser1951}. Birch and Singmaster subsequently developed related product constructions, with unique factorization providing the required distinctness \cite{BirchSingmaster1984}. Sanna gave a later elementary proof of \(\phi(n)>\pi(n)\) for \(n\ge91\), based on Bonse's inequality, and his argument again contains a direct multiplicative mechanism for producing composite totatives \cite{Sanna2012}. Abhijit and Reddy explicitly introduced the function \(\widetilde\phi(n)\) counting nonprime units, which is exactly \(B(n)\) in the present notation, and studied it by further direct product constructions \cite{AbhijitReddy2019}. These results provide the natural arithmetic background for asking not merely whether nonprime totatives eventually dominate prime ones, but how efficiently such domination can be realized by an explicit construction.

The extremal scale is already determined by classical analytic number theory. Landau's minimal-order theorem and the prime number theorem imply
\[
\liminf_{n\to\infty}
\frac{\phi(n)}{\pi(n)}
\frac{\log\log n}{\log n}
=
e^{-\gamma},
\]
where \(\gamma\) is Euler's constant
\cite[Theorem~3.4.2]{LagariasEulerConstant}
\cite[(27.2.3)]{DLMF}.
Moreover, \(A(n)\sim\pi(n)\sim n/\log n\) and \(B(n)\sim\phi(n)\), so that
\[
\liminf_{n\to\infty}
\frac{B(n)}{A(n)}
\frac{\log\log n}{\log n}
=
e^{-\gamma}.
\]
Thus both the scale \(\log n/\log\log n\) and the leading constant \(e^{-\gamma}\) are classical. The question addressed here is different: can this sharp extremal order be realized directly inside the totatives by a single explicit, collision-free family constructed from the prime totatives themselves?

Our construction answers this question affirmatively. Write \(q_1<q_2<\cdots<q_{A(n)}\) for the prime totatives of \(n\). We form squarefree multipliers \(s\) from suitably large \(q_i\)'s and complete each retained multiplier by a larger prime totative \(q_j\). The completion is chosen so that \(q_j\) is the unique largest prime factor of \(sq_j\). Consequently the representation can be recovered from the resulting integer, and the parametrization is injective. The size restriction on the multipliers simultaneously leaves sufficiently many completion primes and keeps every constructed product below \(n\). The main result is the following.

\begin{maintheorem}
For all sufficiently large \(n\), the explicitly defined family \(\mathcal C_n\) is an injectively parametrized family of composite totatives of \(n\) satisfying
\[
\frac{\lvert\mathcal C_n\rvert}{A(n)}
\ge
\left(e^{-\gamma}-o(1)\right)
\frac{\log A(n)}{\log\log A(n)}.
\]
For the same family, with the same defining parameters, there exist absolute effectively computable constants \(C_0,A_0>0\) such that, whenever \(A(n)\ge A_0\),
\[
\frac{\lvert\mathcal C_n\rvert}{A(n)}
\ge
\left(e^{-\gamma}-C_0(\log\log A(n))^{-1/4}\right)
\frac{\log A(n)}{\log\log A(n)}.
\]
Moreover, \(e^{-\gamma}\) is the largest possible leading constant in the following uniform comparison: if \(\mathcal D_n\) is any family of integers counted by \(B(n)\) such that, for some real \(c\),
\[
\frac{\lvert\mathcal D_n\rvert}{A(n)}
\ge
(c-o(1))
\frac{\log A(n)}{\log\log A(n)}
\]
for all sufficiently large \(n\), then \(c\le e^{-\gamma}\).
\end{maintheorem}

Since \(A(n)\sim n/\log n\), we have
\(\log A(n)/\log\log A(n)\sim\log n/\log\log n\).
Consequently, the first bound may equivalently be written as
\(\lvert\mathcal C_n\rvert/A(n)\ge(e^{-\gamma}-o(1))\log n/\log\log n\),
so the construction attains the classical \(n\)-scale with the same sharp leading constant \(e^{-\gamma}\).

The sharp constant arises from the reciprocal distribution of the squarefree multipliers. In the limiting reciprocal model, the relevant law is governed by the perpetuity \(T\), whose distribution is that of \(U(1+T')\), where \(U\) is uniform on \([0,1]\) and \(T'\) is an independent copy of \(T\); the resulting counting constant is \(\mathbb E(1+T)^{-1}=e^{-\gamma}\).

For the effective form, we retain the same construction and the same defining parameters, and use a uniform Wasserstein estimate for the conditional multiplier law, together with the bounded density of the perpetuity, to control the moving retention boundary. Since every member of \(\mathcal C_n\) is counted by \(B(n)\), the construction yields the corresponding sharp lower bounds for \(B(n)/A(n)\) and \(\phi(n)/\pi(n)\). Classical minimal-order theory and the prime number theorem show that the leading constant \(e^{-\gamma}\) is optimal, while the construction realizes it explicitly.

A second, independent question concerns exact eventual dominance. For \(k\ge1\), let \(N_k\) be the least integer such that \(\phi(n)>k\pi(n)\) for every \(n\ge N_k\), and let \(M_k\) be the least integer such that \(B(n)>kA(n)\) for every \(n\ge M_k\). We prove the structural inequality \(M_k\le N_{k+1}\) for every \(k\ge1\), determine \(N_1,\ldots,N_6\) and \(M_1,\ldots,M_5\) exactly, and, for the computed levels, obtain
\[
M_k=N_{k+1}\qquad(1\le k\le5).
\]

The proof combines primorial extremality, a radical--support reduction that turns the remaining bounded ranges into finite exhaustive searches, and explicit analytic estimates that close the tails. The identity
\[
B(n)-kA(n)
=
\phi(n)-(k+1)\pi(n)+(k+1)\omega(n)
\]
shows that the equality of the two threshold sequences is not formal. The five exact coincidences lead naturally to the conjecture \(M_k=N_{k+1}\) for \(k\ge1\).

The paper is organized as follows. Section~\ref{sec:construction} contains the main constructive argument. Subsection~\ref{subsec:sharp-construction} develops the collision-free prime-totative construction and proves the sharp asymptotic lower bound, while Subsection~\ref{subsec:effective} establishes the effective refinement and the classical optimality comparison. Section~\ref{sec:thresholds} turns to the full ratios, develops the primorial and support reductions, determines the exact thresholds above, and concludes with the structural identity and conjecture.

\section{Prime Totatives: Sharp Construction and Effective Refinement}\label{sec:construction}
\subsection{Sharp Construction and Main Result}\label{subsec:sharp-construction}

For \(n\ge2\), a totative of \(n\) is an integer \(m\) with \(1\le m\le n\) and \((m,n)=1\). A prime totative is a totative that is prime, and a nonprime totative is a totative that is not prime. Thus \(1\) is included in the count of nonprime totatives.

Define
\[
\begin{aligned}
A(n)&=\#\{p\le n:p\text{ is prime and }p\nmid n\},\\
B(n)&=\#\{m\le n:(m,n)=1,\ m\text{ is not prime}\}.
\end{aligned}
\]
The totatives split into their prime and nonprime members, while the primes at most \(n\) split into those dividing \(n\) and those not dividing \(n\). Hence \(\phi(n)=A(n)+B(n)\) and \(\pi(n)=A(n)+\omega(n)\). Whenever \(A(n)>0\),
\begin{equation}\label{eq:ratio-decomposition}
\frac{\phi(n)}{\pi(n)}
=
\frac{1+B(n)/A(n)}
     {1+\omega(n)/A(n)}.
\end{equation}

When \(A(n)>0\), write \(q_1<q_2<\cdots<q_{A(n)}\) for the prime totatives of \(n\).

\begin{lemma}\label{lem:denominator}
For every \(n\ge11\), one has \(A(n)\ge\omega(n)+1\). Moreover,
\(\omega(n)/A(n)\to0\) as \(n\to\infty\).
\end{lemma}
\begin{proof}
Put \(r=\omega(n)\). Since \(n\) is divisible by \(r\) distinct primes,
\[
n\ge p_r^\#\ge\prod_{j=1}^r(j+1)=(r+1)!.
\]
The inequality \(A(n)\ge r+1\) is equivalent to \(\pi(n)\ge2r+1\).

If \(r\le2\) and \(n\ge11\), then \(\pi(n)\ge5\). For \(3\le r\le6\), one has
\(p_7=17<30\), \(p_9=23<210\), \(p_{11}=31<2310\), and
\(p_{13}=41<30030\), giving \(p_{2r+1}<p_r^\#\le n\).

For \(r\ge7\), Bertrand's postulate gives \(p_m<2^m\) for \(m\ge2\). Since \(8!>2^{15}\), induction yields \((r+1)!>2^{2r+1}>p_{2r+1}\), and therefore \(\pi(n)\ge2r+1\).

For the asymptotic assertion, fix a positive integer \(h\). As \(r\to\infty\), factorial growth and Bertrand's postulate eventually give \((r+1)!>2^{(h+1)r}>p_{(h+1)r}\). Thus \(\pi(n)\ge(h+1)r\), so \(A(n)=\pi(n)-r\ge hr\). Since \(h\) is arbitrary, \(r/A(n)\to0\) along every sequence on which \(r\to\infty\). If \(r\) remains bounded while \(n\to\infty\), then \(\pi(n)\to\infty\), and hence \(A(n)=\pi(n)-r\to\infty\). The conclusion follows. \end{proof}

In particular, \(A(n)\to\infty\) as \(n\to\infty\).

\begin{lemma}\label{lem:prime-totative-location}
Let \(r=\omega(n)\) and \(A=A(n)\). Whenever \(A>0\),
\[
p_i\le q_i\le p_{i+r}\qquad(1\le i\le A),
\qquad
p_{A+r}\le n<p_{A+r+1}.
\]
As \(n\to\infty\), one has \(r=O\!\left(\frac{\log A}{\log\log A}\right)\). Moreover, for all sufficiently large \(n\), \(n>A\log A\).
\end{lemma}
\begin{proof}
The inequality \(p_i\le q_i\) holds because \(q_i\) is the \(i\)-th member of a subset of the primes. Among the first \(i+r\) primes, at most \(r\) divide \(n\), so at least \(i\) are prime totatives; hence \(q_i\le p_{i+r}\). Since \(A+r=\pi(n)\), the remaining assertion about the position of \(n\) follows immediately.

We use the bounds recorded by Rosser and Schoenfeld in the corollary to their Theorem 3, equations (3.12)–(3.13):
\[
m\log m<p_m\quad(m\ge1),\qquad
p_m<m(\log m+\log\log m)\quad(m\ge6)
\]
\cite[Corollary to Theorem 3, (3.12)--(3.13)]{RosserSchoenfeld1962}.

By Lemma~\ref{lem:denominator}, \(r\le A-1\) for all sufficiently large \(n\). Hence
\[
n<p_{A+r+1}\le p_{2A}
<2A\bigl(\log(2A)+\log\log(2A)\bigr).
\]
Together with \(n\ge(r+1)!\), this gives \(\log((r+1)!)\ll\log A\). If \(r\) remains bounded, the asserted estimate is immediate because \(A\to\infty\). If \(r\to\infty\), factorial growth gives \(r\log r\ll\log A\); inverting the increasing function \(x\mapsto x\log x\) yields \(r=O\!\left(\frac{\log A}{\log\log A}\right)\).

Finally, the lower indexed-prime bound gives \(n\ge p_{A+r}>(A+r)\log(A+r)\ge A\log A\) for all sufficiently large \(n\). \end{proof}

The construction begins with a nonempty squarefree multiplier \(s\) formed from sufficiently large prime totatives, and then completes it by a further prime totative \(q_j\). The completion index is required to satisfy \(j>i(s)\), where \(i(s)\) records the largest prime-totative index occurring in \(s\); thus \(q_j\) becomes the unique largest prime factor of \(sq_j\), making the representation recoverable. The retention condition on \(s\,i(s)\) leaves enough room for admissible completions, while the endpoint \(J_s\) is chosen so that every permitted completion remains below \(n\). The lower cutoff \(R\) also keeps the construction away from the initial prime totatives, which will be important in the reciprocal-product estimates. We now make these choices precise.

For sufficiently large \(n\), put
\[
R=
\left\lceil
\frac{\log A(n)}
     {\sqrt{\log\log A(n)}}
\right\rceil,
\qquad
\delta_A=
\min\left\{
\frac12,\,
(\log\log A(n))^{-1/4}
\right\}.
\]
If \(s\) is a nonempty squarefree product of primes among \(q_R,q_{R+1},\ldots,q_{A(n)}\), put \(i(s)=\max\{i:q_i\mid s\}\). For every such \(s\) satisfying \(s\,i(s)\le A(n)^{1-\delta_A}\), define
\[
J_s=
\left\lfloor
\frac{A(n)\log A(n)}
{\left(1+\dfrac{\omega(n)}R\right)
 s\left(\log\dfrac{A(n)}s+3\log\log A(n)\right)}
\right\rfloor.
\]

\begin{definition}\label{def:construction} For sufficiently large \(n\), define
\[
\mathcal C_n
=
\left\{
sq_j:
\begin{array}{l}
s\ \text{is a nonempty squarefree product of }
q_R,\ldots,q_{A(n)},\\[1mm]
s\,i(s)\le A(n)^{1-\delta_A},\\[1mm]
i(s)<j\le J_s
\end{array}
\right\}.
\]

\end{definition}

\begin{proposition}\label{prop:construction-validity}
For all sufficiently large \(n\), every nonempty squarefree product \(s\) of \(q_R,\ldots,q_{A(n)}\) satisfying \(s\,i(s)\le A(n)^{1-\delta_A}\) also satisfies \(i(s)<J_s\le A(n)\). For every \(i(s)<j\le J_s\), the integer \(sq_j\) is a composite totative of \(n\) and satisfies \(sq_j<n\). Moreover, the parametrization \((s,j)\longmapsto sq_j\) is injective.
\end{proposition}
\begin{proof}
Write \(A=A(n)\), \(r=\omega(n)\), \(X=\log A\), and \(Y=\log X\), and set \(\kappa=1+r/R\) and \(D_s=\log(A/s)+3Y\).
We first verify that every retained multiplier has a nonempty admissible completion interval contained in the available range of prime totatives.
By Lemma~\ref{lem:prime-totative-location}, \(r/R=O(Y^{-1/2})=o(1)\). Also \(\delta_A\to0\), \(\delta_AX\to\infty\), and \(A^{\delta_A}\to\infty\).

If \(s\,i(s)\le A^{1-\delta_A}\), then \(A/(s\,i(s))\ge A^{\delta_A}\) and \(D_s\le X+3Y\). Consequently
\[
\frac{AX}{\kappa sD_s\,i(s)}
\ge
\frac{A^{\delta_A}X}{\kappa(X+3Y)}
\longrightarrow\infty,
\]
so \(J_s>i(s)\) uniformly over the retained multipliers.

The same retention condition gives \(s\le A^{1-\delta_A}/i(s)\le A^{1-\delta_A}/R\) and \(D_s\ge\delta_AX+\log R\).
Since every multiplier is nonempty, \(s\ge q_R\). The lower indexed-prime estimate gives \(q_R>R\log R\), and therefore \(\delta_A q_R\to\infty\). Hence
\[
\frac{J_s}{A}
\le
\frac{X}{\kappa sD_s}
\le
\frac1{\kappa\delta_A s}
=o(1)
\]
uniformly, so \(J_s\le A\) for all sufficiently large \(n\).

It remains to verify that every permitted completion produces an integer below \(n\). Let \(i(s)<j\le J_s\). Since \(j\ge R\), one has \(r+j\le(1+r/R)j=\kappa j\le AX/(sD_s)\).
Set \(m_s=AX/(sD_s)\). From \(D_s\ge\delta_AX+\log R\) and the eventual identity \(\delta_A=Y^{-1/4}\), one has \(\log D_s\ge Y-\tfrac14\log Y\) for all sufficiently large \(n\). Since \(\log m_s=\log(A/s)+Y-\log D_s\), it follows that \(\log m_s+\log\log m_s<D_s\) uniformly over the retained multipliers. Lemma~\ref{lem:prime-totative-location} and the Rosser--Schoenfeld upper bound therefore give
\[
q_j
\le p_{r+j}
<
m_s(\log m_s+\log\log m_s)
<
\frac{AX}{s}.
\]
Thus \(sq_j<AX<n\).
The remaining assertions follow directly from the prime-factor structure of the construction.
Every prime factor of \(s\), and also \(q_j\), is a prime totative of \(n\), so \((sq_j,n)=1\). Since \(s\ne1\) and \(j>i(s)\), the prime \(q_j\) does not divide \(s\); hence \(sq_j\) is composite. Every prime factor of \(s\) is smaller than \(q_j\), so \(q_j\) is the unique largest prime factor of \(sq_j\). If \(sq_j=s'q_{j'}\), unique factorization first gives \(q_j=q_{j'}\), hence \(j=j'\), and then \(s=s'\). The parametrization is injective. \end{proof}

\begin{theorem}\label{thm:sharp-construction} As \(n\to\infty\),
\[
\frac{|\mathcal C_n|}{A(n)}
\ge
\left(e^{-\gamma}-o(1)\right)
\frac{\log A(n)}{\log\log A(n)}.
\]

\end{theorem}
\begin{proof}
Write \(A=A(n)\), \(r=\omega(n)\), \(X=\log A\), and \(Y=\log X\). Then \(R=\lceil X/\sqrt Y\rceil\), and Lemma~\ref{lem:prime-totative-location} gives \(r/R=O(Y^{-1/2})=o(1)\).
We begin by determining the reciprocal-product scale of the available prime totatives.
For \(R\le i\le A\), Lemma~\ref{lem:prime-totative-location} and the Rosser--Schoenfeld bounds imply
\[
\frac1{(i+r)(\log(i+r)+\log\log(i+r))}
<
\frac1{q_i}
<
\frac1{i\log i}.
\]
Let \(f(x)=1/(x\log x)\). Uniformly for \(R\le m\le A\),
\[
0\le
\sum_{i=R}^{m}\bigl(f(i)-f(i+r)\bigr)
\le
\sum_{i=R}^{R+r-1}f(i)
\ll\frac{r}{R\log R}=o(1).
\]
Moreover,
\[
0\le
\frac1{k\log k}
-
\frac1{k(\log k+\log\log k)}
\le
\frac{\log\log k}{k(\log k)^2},
\]
and the sum of the last expression over \(k\ge R\) is \(O((\log\log R+1)/\log R)=o(1)\).
An integral comparison for \(1/(x\log x)\) therefore gives, uniformly for \(R\le m\le A\),
\[
\begin{aligned}
\sum_{i=R}^{m}\frac1{q_i}
&=\log\log m-\log\log R+o(1),\\
\sum_{i=R}^{A}\frac1{q_i^2}
&\le\sum_{i=R}^{\infty}\frac1{(i\log i)^2}=o(1).
\end{aligned}
\]

Set \(Z_m=\prod_{i=R}^{m}\left(1+1/q_i\right)\).
The preceding estimates and \(\log(1+x)=x+O(x^2)\) yield the reciprocal-product asymptotics
\begin{equation}\label{eq:reciprocal-product-asymptotics}
Z_A=(1+o(1))\frac{X}{Y},
\qquad
\frac{Z_{\lfloor A^u\rfloor}}{Z_A}\longrightarrow u
\quad(0<u\le1\ \text{fixed}).
\end{equation}
We next identify the joint limiting law governing a reciprocal-random multiplier and its largest selected prime factor.
Consider all squarefree products \(S\) of \(q_R,\ldots,q_A\), including the empty product, with reciprocal probability
\[
\mathbb P(S=s)=\frac1{Z_As}.
\]
For probabilistic bookkeeping only, set \(i(1)=1\) and \(V_n=\log S/X\).
For fixed \(t\ge0\) and \(0<u\le1\), the event \(i(S)\le\lfloor A^u\rfloor\) restricts the squarefree product to the first \(\lfloor A^u\rfloor-R+1\) available prime totatives. Hence the exact finite Euler-product identity is
\[
\mathbb E\!\left[
e^{-tV_n};
\frac{\log i(S)}X\le u
\right]
=
\frac1{Z_A}
\prod_{i=R}^{\lfloor A^u\rfloor}
\left(1+q_i^{-1-t/X}\right).
\]

Define, for \(t\ge0\),
\[
L(t)=
\exp\left(
-\int_0^1\frac{1-e^{-tx}}x\,dx
\right).
\]
The function \((1-e^{-tx})/x\) extends continuously to \(x=0\), with value \(t\). The reciprocal estimate above gives
\[
\sum_{i=R}^{\lfloor A^u\rfloor}
\left|
\frac1{q_i}-\frac1{i\log i}
\right|
=o(1).
\]
Furthermore, uniformly for \(i\ge R\), \(0\le\log q_i-\log i\ll\log\log i\). Since \(x\mapsto1-e^{-tx}\) is Lipschitz,
\[
\sum_{i=R}^{\lfloor A^u\rfloor}
\frac{
\left|
(1-e^{-t\log q_i/X})-(1-e^{-t\log i/X})
\right|}
{q_i}
\ll_t
\frac1X
\sum_{i=R}^{\lfloor A^u\rfloor}
\frac{\log\log i}{i\log i}
=
O_t\!\left(\frac{Y^2}{X}\right)
=o(1).
\]

To pass from the resulting sum over \(i\) to the integral, set
\[
x_i=\frac{\log i}{X},\qquad
h_t(x)=
\begin{cases}
\dfrac{1-e^{-tx}}x,&x>0,\\[2mm]
t,&x=0.
\end{cases}
\]
Then \(h_t\) is continuous on \([0,1]\), and \(x_{i+1}-x_i=1/(iX)+O(1/(i^2X))\). Since \(h_t\) is bounded on \([0,1]\) and \(\sum_{i\ge R}1/(i^2X)=o(1)\), replacing \(1/(iX)\) by \(x_{i+1}-x_i\) introduces only \(o(1)\). Since \((1-e^{-t\log i/X})/(i\log i)=h_t(x_i)/(iX)\), the resulting expression is an ordinary Riemann sum. Together with the preceding comparison, this gives
\[
\begin{aligned}
\sum_{i=R}^{\lfloor A^u\rfloor}
\frac{1-e^{-t\log i/X}}{i\log i}
&\longrightarrow
\int_0^u\frac{1-e^{-tx}}x\,dx,\\
\sum_{i=R}^{\lfloor A^u\rfloor}
\frac{1-e^{-t\log q_i/X}}{q_i}
&\longrightarrow
\int_0^u\frac{1-e^{-tx}}x\,dx.
\end{aligned}
\]
Expanding the logarithm of the exact Euler product, with total quadratic error \(o(1)\), gives
\[
\mathbb E\!\left[
e^{-tV_n};
\frac{\log i(S)}X\le u
\right]
\longrightarrow
uL(tu).
\]

Let \(U_1,U_2,\ldots\) be independent random variables uniformly distributed on \([0,1]\), and set
\[
T=U_1+U_1U_2+U_1U_2U_3+\cdots.
\]
By monotone convergence, \(\mathbb ET=\sum_{k\ge1}2^{-k}=1\), so the series converges to a finite value almost surely. Splitting off its first factor gives the perpetuity relation
\[
T\overset d=U(1+T'),
\]
where \(U\) is uniform on \([0,1]\), \(T'\overset d=T\), and \(U,T'\) are independent.

If \(F(t)=\mathbb E e^{-tT}\), conditioning on \(U\), followed by the change of variables \(y=tu\) and differentiation, gives
\[
\begin{aligned}
F(t)&=\int_0^1e^{-tu}F(tu)\,du,\\
tF(t)&=\int_0^t e^{-y}F(y)\,dy,\\
\frac{F'(t)}{F(t)}&=-\frac{1-e^{-t}}t.
\end{aligned}
\]
Since \(F(0)=1\), one has \(F=L\). In particular,
\[
\mathbb E\!\left[e^{-tU(1+T)};U\le u\right]
=
\int_0^u e^{-tx}L(tx)\,dx
=
uL(tu),
\]
the last equality following from \(\frac{d}{dx}\{xL(tx)\}=e^{-tx}L(tx)\).

It remains to identify the joint limit. First, \(V_n\) is tight. The function
\(h(x)=\log x/(x+1)\) satisfies
\(h'(x)=(1+1/x-\log x)/(x+1)^2<0\) for \(x\ge4\).
Since \(R\to\infty\) and \(q_i>i\log i\), it follows uniformly for
\(i\ge R\), once \(n\) is sufficiently large, that
\(\log q_i/(q_i+1)\le \log(i\log i)/(i\log i)\). Hence
\[
\begin{aligned}
\mathbb EV_n
&=\frac1X\sum_{i=R}^{A}\frac{\log q_i}{q_i+1}\\
&\le
\frac1X\sum_{i=R}^{A}
\left(
\frac1i+\frac{\log\log i}{i\log i}
\right)
=O(1).
\end{aligned}
\]
Thus \(V_n\) is tight, while \(\log i(S)/X\in[0,1]\). It follows that the joint laws are tight, and every subsequence has a further weakly convergent subsequence.

Taking \(t=0\) in the mixed-transform limit gives
\(\mathbb P(\log i(S)/X\le u)\to u\). Consider a weakly convergent further subsequence along which
\[
\left(V_n,\frac{\log i(S)}X\right)
\Rightarrow(\widetilde V,\widetilde U).
\]
The second marginal is uniform and therefore atomless. For fixed \(t\ge0\) and
\(0<u\le1\), the bounded function
\((v,w)\mapsto e^{-tv}\mathbf1_{\{w\le u\}}\) is discontinuous only on \(w=u\),
which has limiting probability zero. Hence
\[
\mathbb E[e^{-t\widetilde V};\widetilde U\le u]
=
uL(tu).
\]

For each fixed \(u\), the left-hand side is the Laplace transform of the finite measure
\(E\mapsto\mathbb P(\widetilde V\in E,\widetilde U\le u)\) on \([0,\infty)\).
Such a measure is determined by its Laplace transform. Indeed, push it forward under
\(x\mapsto e^{-x}\), and regard the pushforward as a finite measure on \([0,1]\)
by assigning zero mass to \(0\). Equality of Laplace transforms at the nonnegative
integers is then equality of all polynomial moments on \([0,1]\). By the Weierstrass
approximation theorem, the two finite measures agree. Thus, for every \(u\), the
finite measures associated with the limiting law agree with those associated with
\((U(1+T),U)\). The rectangles \(E\times[0,u]\), with
\(E\subseteq[0,\infty)\) Borel and \(0\le u\le1\), generate the Borel
\(\sigma\)-algebra on \([0,\infty)\times[0,1]\); hence these identities determine
the joint law. Every subsequential limit is therefore the same, and
\begin{equation}\label{eq:joint-limit}
\left(
V_n,\frac{\log i(S)}X
\right)
\Rightarrow
\bigl(U(1+T),U\bigr).
\end{equation}

We use this limiting law to obtain a lower bound for the number of retained completions.

Fix \(0<\delta<1\). Since \(\delta_A\to0\), we have \(\delta_A<\delta\) for all sufficiently large \(n\). Hence every nonempty squarefree product \(s\) satisfying \(s\,i(s)\le A^{1-\delta}\) also satisfies \(s\,i(s)\le A^{1-\delta_A}\), so the endpoint \(J_s\) from Definition~\ref{def:construction} is already defined for every multiplier considered below.

Within this proof, let \(\mathcal C_{n,\delta}\) denote the family obtained from Definition~\ref{def:construction} by replacing \(\delta_A\) with \(\delta\), leaving \(R\) and \(J_s\) unchanged. The argument of Proposition~\ref{prop:construction-validity}, with \(\delta\) fixed, shows that its completion intervals are valid and that the resulting parametrization is injective.

For the remainder of the fixed-\(\delta\) argument, put \(\kappa=1+r/R\) and \(D_s=\log(A/s)+3Y\). By injectivity, \(|\mathcal C_{n,\delta}|=\sum_{\text{retained }s}(J_s-i(s))\). Using \(\lfloor x\rfloor\ge x-1\) yields the principal counting inequality
\[
\frac{|\mathcal C_{n,\delta}|}{A}
\ge
\frac{X}{\kappa}
\sum_{\text{retained }s}\frac1{sD_s}
-
\frac1A
\sum_{\text{retained }s}\bigl(i(s)+1\bigr).
\]
Retention gives \(i(s)\le A^{1-\delta}/s\), and therefore \(A^{-1}\sum_{\text{retained }s}i(s)\le A^{-\delta}Z_A=o(X/Y)\). The retained multipliers are distinct positive integers not exceeding \(A^{1-\delta}/R\), so their number divided by \(A\) is also \(o(X/Y)\). Hence
\[
\frac{|\mathcal C_{n,\delta}|}{A}
\ge
\frac{X}{\kappa}
\sum_{\text{retained }s}\frac1{sD_s}
-o(X/Y).
\]

The reciprocal probability space contains the empty product, whereas \(\mathcal C_{n,\delta}\) uses only nonempty multipliers. For \(s\ne1\), the retention condition is equivalent to \(V_n+\log i(S)/X\le1-\delta\), while \(X/D_s=(1-V_n+3Y/X)^{-1}\). Thus the exact correction is
\[
\begin{aligned}
\frac{X}{\kappa}
\sum_{\substack{s\ne1\\ s\ {\rm retained}}}
\frac1{sD_s}
&=
\frac{Z_A}{\kappa}
\mathbb E\!\left[
\frac{
\mathbf1_{\{
V_n+\log i(S)/X\le1-\delta
\}}
}{
1-V_n+3Y/X
}
\right]\\
&\quad-
\frac{X}{\kappa(X+3Y)}.
\end{aligned}
\]
The last term is \(O(1)=o(X/Y)\).

On the event in the expectation one has \(V_n\le1-\delta\); hence both \(1-V_n\) and \(1-V_n+3Y/X\) are at least \(\delta\). Replacing \((1-V_n+3Y/X)^{-1}\) by \((1-V_n)^{-1}\) therefore changes the expectation by \(O_\delta(Y/X)=o(1)\), uniformly in \(n\). The resulting test function equals \((1-v)^{-1}\) when \(v+u\le1-\delta\) and \(0\) otherwise. It is bounded by \(1/\delta\) and is continuous away from the boundary \(v+u=1-\delta\). Under the limiting law this boundary is the event \(U(T+2)=1-\delta\), which has probability zero because, conditional on \(T\), \(U(T+2)\) has a continuous distribution. By \eqref{eq:joint-limit},
\[
\begin{aligned}
\mathbb E\!\left[
\frac{
\mathbf1_{\{
V_n+\log i(S)/X\le1-\delta
\}}
}{
1-V_n+3Y/X
}
\right]
&\longrightarrow c_\delta,\\
c_\delta
&=
\mathbb E\!\left[
\frac{
\mathbf1_{\{U(T+2)\le1-\delta\}}
}{
1-U(T+1)
}
\right].
\end{aligned}
\]
Combining \eqref{eq:reciprocal-product-asymptotics} with \(\kappa\to1\) yields
\[
\frac{|\mathcal C_{n,\delta}|}{A}
\ge
(c_\delta+o(1))\frac{X}{Y}.
\]

It remains to evaluate the limiting constant and then remove the fixed retention parameter. Conditioning on \(T\), and then letting \(\delta\downarrow0\), gives
\[
\begin{aligned}
c_\delta
&=
\mathbb E\!\left[
\frac{
\log\!\left(
\dfrac{T+2}{1+\delta(T+1)}
\right)}
{T+1}
\right],\\
c_\delta&\uparrow c_0
=
\mathbb E\frac{\log(T+2)}{T+1}.
\end{aligned}
\]
The perpetuity relation gives \(c_0=\mathbb E(1+T)^{-1}\).
By Tonelli's theorem and \(F=L\),
\[
\mathbb E\frac1{1+T}
=
\int_0^\infty e^{-t}L(t)\,dt
=
\lim_{t\to\infty}tL(t),
\]
because \((tL(t))'=e^{-t}L(t)\).

Finally, \(-\log L(t)=\int_0^t(1-e^{-y})/y\,dy\). NIST DLMF, equation (5.9.18), gives the integral representation
\[
\gamma
=
\int_0^1\frac{1-e^{-y}}y\,dy
-
\int_1^\infty\frac{e^{-y}}y\,dy
\]
\cite[(5.9.18)]{DLMF}.
It follows that \(\int_0^t(1-e^{-y})/y\,dy=\log t+\gamma+o(1)\), and therefore
\[
c_0=e^{-\gamma}.
\]
Moreover,
\[
0\le e^{-\gamma}-c_\delta
=
\mathbb E
\frac{\log(1+\delta(T+1))}{T+1}
\le\delta.
\]

We return to the family \(\mathcal C_n\) of Definition~\ref{def:construction}. Given \(\varepsilon>0\), choose a fixed \(0<\delta_0<\varepsilon/2\). Since \(\delta_A\to0\), eventually \(\delta_A<\delta_0\), and hence \(s\,i(s)\le A^{1-\delta_0}\) implies \(s\,i(s)\le A^{1-\delta_A}\). The endpoint \(J_s\) does not depend on the retention parameter, so \(\mathcal C_{n,\delta_0}\subseteq\mathcal C_n\) for all sufficiently large \(n\). The fixed-\(\delta_0\) estimate and \(c_{\delta_0}\ge e^{-\gamma}-\delta_0\) give \(|\mathcal C_n|/A\ge(e^{-\gamma}-\varepsilon)X/Y\) eventually. Since \(\varepsilon>0\) is arbitrary,
\[
\frac{|\mathcal C_n|}{A(n)}
\ge
\left(e^{-\gamma}-o(1)\right)
\frac{\log A(n)}{\log\log A(n)}.
\]
\end{proof}

\begin{corollary}\label{cor:ratio-consequences}
As \(n\to\infty\),
\[
\begin{aligned}
\frac{B(n)}{A(n)}
&\ge
\left(e^{-\gamma}-o(1)\right)
\frac{\log A(n)}{\log\log A(n)},\\
\frac{\phi(n)}{\pi(n)}
&=
\frac{B(n)}{A(n)}(1+o(1)).
\end{aligned}
\]
Consequently, \(\phi(n)/\pi(n)\ge (e^{-\gamma}-o(1))\log A(n)/\log\log A(n)\), and both \(B(n)/A(n)\) and \(\phi(n)/\pi(n)\) tend to infinity.
\end{corollary}

\begin{proof}
By Proposition~\ref{prop:construction-validity}, every element of \(\mathcal C_n\) is counted by \(B(n)\), so \(|\mathcal C_n|\le B(n)\); Theorem~\ref{thm:sharp-construction} therefore gives the first estimate. Since \(\log A(n)/\log\log A(n)\to\infty\), we have \(B(n)/A(n)\to\infty\). Moreover, Lemma~\ref{lem:denominator} gives \(\omega(n)/A(n)\to0\). Hence \eqref{eq:ratio-decomposition} yields \(\phi(n)/\pi(n)=(B(n)/A(n))(1+o(1))\), and the remaining assertions follow.
\end{proof}

\subsection{Quantitative Refinement and Classical Optimality}\label{subsec:effective}

Theorem~\ref{thm:sharp-construction} gives the sharp leading constant for the explicit family \(\mathcal C_n\), but only with an unspecified \(o(1)\) error. We seek an effective quantitative estimate for the same family, with exactly the same defining parameters, by controlling the approximation mechanism already used in the sharp asymptotic argument rather than constructing a different family.

The target constant is not merely an artifact of that argument. As recalled below from Landau's minimal-order theorem and the prime number theorem, \(e^{-\gamma}\) is the classical uniform leading constant for the full ratios at the scale \(\log n/\log\log n\). Thus the quantitative problem is to measure how effectively the explicit construction approaches this classical ceiling. The main additional input is a quantitative stability estimate for the reciprocal multiplier law, together with control of the moving retention boundary.

The stability argument is organized by the largest selected prime-totative factor. Conditional on its index, the factors below it retain a reciprocal product law whose normalized logarithm is modeled by the perpetuity \(T\) appearing in the proof of Theorem~\ref{thm:sharp-construction}. A uniform Wasserstein estimate quantifies this approximation, while the anti-concentration of the perpetuity controls the discontinuity created by the moving retention boundary. These are the two ingredients needed to turn the qualitative limiting argument into an effective estimate.

For real-valued random variables \(V,W\) with distribution functions \(F_V,F_W\), write \(d_{\mathrm K}(V,W)=\sup_{x\in\mathbb R}|F_V(x)-F_W(x)|\) for the Kolmogorov distance. When \(V\) and \(W\) have finite first moments, write \(d_1(V,W)=\inf \mathbb E|V'-W'|\), where the infimum is taken over all couplings \((V',W')\) with \(V'\overset d=V\) and \(W'\overset d=W\).

\begin{lemma}\label{lem:wasserstein-stability}
Let \(n\) be sufficiently large, and set
\[
\begin{aligned}
A&=A(n),\qquad X=\log A,\qquad Y=\log X,\qquad
R=\left\lceil\frac{X}{\sqrt Y}\right\rceil,\\
\ell_m&=\log q_m,\qquad
Z_m=\prod_{i=R}^{m}\left(1+\frac1{q_i}\right)
\quad(R\le m\le A),\qquad Z_{R-1}=1.
\end{aligned}
\]
For every integer \(k\) with \(R<k\le A\), consider the squarefree products \(S\) of \(q_R,\ldots,q_{k-1}\) under the reciprocal law \(\mathbb P(S=s)=1/(Z_{k-1}s)\), and set \(W_R=0\) and \(W_k=\log S/\ell_k\). Let \(T\) be the perpetuity from the proof of Theorem~\ref{thm:sharp-construction}, satisfying \(T\overset d=U(1+T')\) and \(\mathbb ET=1\), where \(U\) is uniform on \([0,1]\), \(T'\overset d=T\), and \(U\) and \(T'\) are independent.

There is an effectively computable absolute constant \(C>0\) such that, uniformly for \(R<k\le A\),
\[
d_1(W_k,T)
\le
C\left(
Y^{-3/2}
+
\sqrt{\frac{\ell_R}{\ell_k}}
\right).
\]
\end{lemma}

\begin{proof}
Put \(r=\omega(n)\). The effective form of Lemma~\ref{lem:prime-totative-location} gives \(r\ll_{\mathrm{eff}}X/Y\) and \(r/R\ll_{\mathrm{eff}}Y^{-1/2}\). We first establish the partial-product estimate required uniformly in \(m\). Within this proof, let \(P_+(x)=\prod_{p\le x}(1+1/p)\). Rosser and Schoenfeld prove
\[
\begin{aligned}
e^\gamma\log x
\left(1-\frac1{2\log^2x}\right)
&<
\prod_{p\le x}\frac p{p-1}
&& (x>1),\\
\prod_{p\le x}\frac p{p-1}
&<
e^\gamma\log x
\left(1+\frac1{2\log^2x}\right)
&& (x\ge286),
\end{aligned}
\]
\cite[Theorem 8, p.~70]{RosserSchoenfeld1962}.
Since \(1+1/p=(p/(p-1))(1-1/p^2)\), we obtain, effectively,
\[
\begin{aligned}
\prod_{p\le x}\left(1-\frac1{p^2}\right)
&=
\frac1{\zeta(2)}
\left(1+O_{\mathrm{eff}}(x^{-1})\right),\\
P_+(x)
&=
\frac{e^\gamma}{\zeta(2)}
\log x
\left(1+O_{\mathrm{eff}}(\log^{-2}x)\right).
\end{aligned}
\]

At the lower endpoint we use the exact identity \(P_+(q_R^-)=P_+(q_R)/(1+1/q_R)\). Indeed, \(q_R\) is a prime totative and therefore does not divide \(n\).
\[
Z_m
=
\frac{P_+(q_m)}{P_+(q_R^-)}
\prod_{\substack{p\mid n\\q_R\le p\le q_m}}
\left(1+\frac1p\right)^{-1}.
\]
The omitted prime divisors satisfy \(0\le\sum_{\substack{p\mid n\\p\ge q_R}}1/p\le r/q_R\).
The indexed-prime inequalities
\[
p_m>m\log m\qquad(m\ge1),
\qquad
p_m<m(\log m+\log\log m)\qquad(m\ge6)
\]
are recorded in the corollary to Rosser and Schoenfeld's Theorem 3 as equations (3.12)--(3.13)
\cite[Corollary to Theorem 3, (3.12)--(3.13)]{RosserSchoenfeld1962}.
They give \(q_R\ge p_R>R\log R\gg X\sqrt Y\). Consequently \(r/q_R\ll_{\mathrm{eff}}Y^{-3/2}\). Also \(\ell_R\asymp Y\), so the Mertens-product error at the lower endpoint is \(O_{\mathrm{eff}}(Y^{-2})\). Hence, uniformly for every \(R\le m\le A\),
\begin{equation}\label{eq:partial-product}
Z_m
=
\frac{\ell_m}{\ell_R}
\left(
1+O_{\mathrm{eff}}(Y^{-3/2})
\right).
\end{equation}

For the remainder of the proof, fix an integer \(k\) with \(R<k\le A\). Let \(J=0\) when \(S=1\), and let \(J=i(S)\) otherwise. Define
\[
A_k=
\begin{cases}
0,&J=0,\\[1mm]
\dfrac{\ell_J}{\ell_k},&J\ge R.
\end{cases}
\]
Unique factorization gives
\[
\begin{aligned}
\mathbb P(J=j)
&=
\frac{Z_{j-1}}{q_jZ_{k-1}}
\qquad(R\le j<k),\\
\mathbb P(J=0)
&=Z_{k-1}^{-1}.
\end{aligned}
\]
Indeed, on the event \(J=j\), the factor \(q_j\) is selected and the factors below \(q_j\) form an arbitrary squarefree product of \(q_R,\ldots,q_{j-1}\). Conditionally on \(J=j\), those lower factors therefore retain exactly the reciprocal law defining \(W_j\). Hence
\[
W_k\overset d=
\begin{cases}
0,&J=0,\\[1mm]
\dfrac{\ell_J}{\ell_k}(1+W_J),&J\ge R.
\end{cases}
\]

We next quantify the distribution of \(A_k\). Put \(x_m=\ell_m/\ell_k\) for \(R\le m\le k\), so that \(x_k=1\). At the support points \(R\le m<k\), \eqref{eq:partial-product} gives
\[
\mathbb P(A_k\le x_m)
=
\frac{Z_m}{Z_{k-1}}
=
\frac{\ell_m}{\ell_{k-1}}
\left(
1+O_{\mathrm{eff}}(Y^{-3/2})
\right).
\]

The mesh requires separate control. From \(q_{m+1}\le p_{m+1+r}\) and \(q_m\ge p_m\), we have \(\ell_{m+1}-\ell_m\le\log(p_{m+1+r}/p_m)\). Since \(m\ge R\), one has \(r/m\le r/R\ll_{\mathrm{eff}}Y^{-1/2}\). The indexed-prime estimates therefore yield \(\log(p_{m+1+r}/p_m)\ll_{\mathrm{eff}}Y^{-1/2}+(\log Y)/Y\ll_{\mathrm{eff}}Y^{-1/2}\). As \(\ell_k\ge\ell_R\gg Y\),
\[
\max_{R\le m<k}
\frac{\ell_{m+1}-\ell_m}{\ell_k}
\ll_{\mathrm{eff}}Y^{-3/2}.
\]

At the lower endpoint, \(\mathbb P(A_k=0)=Z_{k-1}^{-1}=O_{\mathrm{eff}}(\ell_R/\ell_k+Y^{-3/2})\). Writing \(b_k=\ell_R/\ell_k\), we therefore have, for \(0<x<x_R\),
\[
|\mathbb P(A_k\le x)-x|
\le \mathbb P(A_k=0)+x_R
=
O_{\mathrm{eff}}(Y^{-3/2}+b_k).
\]

If \(x_m\le x<x_{m+1}\) with \(R\le m<k\), then
\[
\begin{aligned}
\left|\mathbb P(A_k\le x)-x\right|
&\le
\left|
\frac{Z_m}{Z_{k-1}}
-
\frac{\ell_m}{\ell_{k-1}}
\right|\\
&\quad+
\left|
\frac{\ell_m}{\ell_{k-1}}
-
\frac{\ell_m}{\ell_k}
\right|
+
|x-x_m|.
\end{aligned}
\]
The first term is \(O_{\mathrm{eff}}(Y^{-3/2})\), while the second and third are bounded by the preceding mesh estimate. Thus, letting \(U\) be uniform on \([0,1]\), there is an effectively computable absolute constant \(C_1\) such that
\[
d_{\mathrm K}(A_k,U)
\le
C_1\left(Y^{-3/2}+b_k\right).
\]
For probability laws on the line, the Wasserstein coupling distance satisfies
\[
d_1(A_k,U)
=
\int_0^1
\left|
F_{A_k}(x)-F_U(x)
\right|\,dx
\le d_{\mathrm K}(A_k,U).
\]

Put \(e_k=d_1(W_k,T)\). For each \(R\le j<k\), choose a Wasserstein coupling of \(W_j\) with a variable \(T_j\sim T\), with the conditional \(T_j\)-marginal equal to the same law \(T\) for every \(j\). On the branch \(J=0\), choose the corresponding \(T_J\) with the same law \(T\). Thus \(T_J\) is independent of \(J\), and hence of \(A_k\). Couple \(A_k\) with a uniform random variable \(U\) using auxiliary randomness independent of \(T_J\). It follows that \(U(1+T_J)\) has the perpetuity law \(T\). The branchwise recursion and the triangle inequality therefore give
\[
e_k
\le
\sum_{j=R}^{k-1}
\mathbb P(J=j)
\frac{\ell_j}{\ell_k}e_j
+
2d_1(A_k,U),
\]
because \(\mathbb E(1+T)=2\).

Set \(\varepsilon=Y^{-3/2}\). We prove uniformly in \(k\) that
\[
e_k
\le
C\left(
\varepsilon+\sqrt{b_k}
\right)
\]
for an effectively computable absolute constant \(C\). At the base index, \(W_R=0\), \(e_R=d_1(W_R,T)=\mathbb ET=1\), and \(b_R=1\), so the proposed bound holds at \(k=R\) once \(C\ge1\).

For \(k>R\), assume the bound for smaller indices. Since \((\ell_j/\ell_k)\sqrt{b_j}=\sqrt{b_k}\sqrt{\ell_j/\ell_k}\), we obtain
\[
e_k
\le
C\varepsilon\,\mathbb EA_k
+
C\sqrt{b_k}\,\mathbb E\sqrt{A_k}
+
2C_1(\varepsilon+b_k).
\]
The preceding Kolmogorov estimate, applied to the bounded-variation functions \(x\) and \(\sqrt x\), yields
\[
\mathbb EA_k
=
\frac12+O_{\mathrm{eff}}(\varepsilon+b_k),
\qquad
\mathbb E\sqrt{A_k}
=
\frac23+O_{\mathrm{eff}}(\varepsilon+b_k).
\]
Choose a fixed \(\beta>0\) sufficiently small that, whenever \(\varepsilon+b_k\le\beta\), both expectations are at most \(3/4\). After choosing, for example, \(C\ge8C_1\), the induction closes in this regime.

The claimed bound also holds uniformly in the complementary regime. Once \(Y\) is effectively large enough that \(\varepsilon\le\beta/2\), the condition \(\varepsilon+b_k>\beta\) implies \(b_k\ge\beta/2\). Moreover,
\[
\mathbb EW_k
=
\frac1{\ell_k}
\sum_{i=R}^{k-1}
\frac{\ell_i}{q_i+1}
<
\frac1{\ell_k}
\sum_{p\le q_{k-1}}
\frac{\log p}{p}
<1,
\]
where the last inequality is Rosser and Schoenfeld's equation (3.24)
\cite[(3.24), p.~70]{RosserSchoenfeld1962}.
Since \(\mathbb ET=1\), we have \(e_k\le\mathbb EW_k+\mathbb ET<2\). Enlarging \(C\) so that \(2\le C\sqrt{\beta/2}\) establishes the claimed bound throughout the complementary regime. This proves the lemma.
\end{proof}

The perpetuity representation also supplies the anti-concentration needed below. Conditional on \(T'=t\), \(T=U(1+t)\) is uniform on \([0,1+t]\). Hence \(T\) has density
\[
f_T(x)
=
\mathbb E\left[
\frac{\mathbf1_{\{0\le x\le1+T'\}}}{1+T'}
\right],
\qquad
\|f_T\|_\infty\le1.
\]
We shall use the following one-jump consequence. Suppose \(d_1(W,T)\le e\), let \(a\ge0\), and define
\[
f(w)
=
\begin{cases}
\dfrac1{c-aw},&w\le b,\\[2mm]
0,&w>b,
\end{cases}
\]
where \(c-aw\ge\delta>0\) on the support of \(f\). Away from the jump, \(|f'(w)|=a/(c-aw)^2\le a/\delta^2\). Choose a coupling with \(\mathbb E|W-T|\le e\) and put \(h=\sqrt e\). The event on which the jump can intervene is contained in \(\{|W-T|>h\}\cup\{|T-b|\le h\}\). By Markov's inequality and the density bound above, its probability is at most \(3\sqrt e\). On the complement, \(W\) and \(T\) lie on the same side of the jump. Consequently
\begin{equation}\label{eq:one-jump-smoothing}
|\mathbb Ef(W)-\mathbb Ef(T)|
\le
\frac{3\sqrt e}{\delta}
+
\frac{ae}{\delta^2}.
\end{equation}

\begin{theorem}\label{thm:effective-construction}
There exist absolute effectively computable constants \(C_0,A_0>0\) such that, whenever \(A(n)\ge A_0\),
\[
\frac{|\mathcal C_n|}{A(n)}
\ge
\left(
e^{-\gamma}
-
C_0(\log\log A(n))^{-1/4}
\right)
\frac{\log A(n)}{\log\log A(n)}.
\]
\end{theorem}

\begin{proof}
Write \(A=A(n)\), \(r=\omega(n)\), \(X=\log A\), and \(Y=\log X\). We use exactly the parameters defining \(\mathcal C_n\) in Subsection~\ref{subsec:sharp-construction}: \(R=\lceil X/\sqrt Y\rceil\) and \(\delta_A=\min\{1/2,Y^{-1/4}\}\). After increasing the eventual threshold, we may assume throughout that \(\delta_A=Y^{-1/4}\). The effective form of Lemma~\ref{lem:prime-totative-location} gives \(r\ll_{\mathrm{eff}}X/Y\) and \(r/R\ll_{\mathrm{eff}}Y^{-1/2}\).

Put \(\ell_m=\log q_m\) and \(Z_m=\prod_{i=R}^{m}(1+1/q_i)\) for \(R\le m\le A\), with \(Z_{R-1}=1\). The partial-product estimate \eqref{eq:partial-product} holds uniformly for \(R\le m\le A\).

Put \(\delta=\delta_A=Y^{-1/4}\), and introduce the auxiliary parameter \(\eta=Y^{-2}\). Under the reciprocal probability law on all squarefree products \(S\) of \(q_R,\ldots,q_A\), \(\mathbb P(S=s)=1/(Z_As)\); retain the convention \(J=0\) for the empty product and \(J=i(S)\) otherwise.

Consider first the good conditional sector \(k>\lfloor A^\eta\rfloor\). Then \(k>A^\eta\) and \(\log k>\eta X\), so \(\ell_k\ge\log k>X/Y^2\). Since \(\ell_R\asymp Y\), Lemma~\ref{lem:wasserstein-stability} gives \(d_1(W_k,T)=O_{\mathrm{eff}}(Y^{-3/2})\) uniformly throughout this sector.

For such \(k\), put \(u_k=\log k/X\) and \(a_k=\ell_k/X\), and define the conditional payoff branchwise by
\[
f_{A,k}(w)
=
\begin{cases}
\dfrac1{1-a_k(1+w)+3Y/X},
&
a_k(1+w)+u_k\le1-\delta,\\[3mm]
0,
&
a_k(1+w)+u_k>1-\delta.
\end{cases}
\]
On its support the denominator is at least \(\delta\). Since \(k\le A\), the indexed-prime bounds give \(a_k\le2\) after an effective threshold. Applying \eqref{eq:one-jump-smoothing} with \(e=O_{\mathrm{eff}}(Y^{-3/2})\) therefore gives
\[
\left|
\mathbb E f_{A,k}(W_k)
-
\mathbb E f_{A,k}(T)
\right|
=
O_{\mathrm{eff}}(Y^{-1/2})
\]
uniformly in the good sector.

We next compare the prime-totative scale with the index scale. The indexed-prime bounds give \(|a_k-u_k|=O_{\mathrm{eff}}(Y/X)\) uniformly for \(k>\lfloor A^\eta\rfloor\), while \(a_k\ge u_k>\eta\). The corresponding displacement of the cutoff in the variable \(w\) is \(O_{\mathrm{eff}}((Y/X)/\eta^2)=O_{\mathrm{eff}}(Y^5/X)\). Using the density bound for \(T\), the lower bound \(\delta\) for the denominator, and elementary differentiation on the common support gives a total parameter-perturbation error \(O_{\mathrm{eff}}(Y^{21/4}/X)\).

Define
\[
h_\delta(u)
=
\mathbb E\left[
\begin{cases}
\dfrac1{1-u(T+1)},
&
u(T+2)\le1-\delta,\\[3mm]
0,
&
u(T+2)>1-\delta.
\end{cases}
\right].
\]
The preceding perturbation calculation is precisely
\[
\left|
\mathbb E f_{A,k}(T)
-
h_\delta(u_k)
\right|
=
O_{\mathrm{eff}}\left(\frac{Y^{21/4}}{X}\right)
\]
uniformly in the good sector.

The complementary small-largest-factor sector is \(J\le\lfloor A^\eta\rfloor\). For sufficiently large \(A\), \(A^\eta>R\), so \eqref{eq:partial-product} applies at \(m=\lfloor A^\eta\rfloor\), and
\[
\mathbb P\!\left(J\le\lfloor A^\eta\rfloor\right)
=
\frac{Z_{\lfloor A^\eta\rfloor}}{Z_A}
=
O_{\mathrm{eff}}\left(\eta+Y^{-3/2}\right)
=
O_{\mathrm{eff}}(Y^{-3/2}).
\]
The finite payoff is bounded by \(\delta^{-1}=Y^{1/4}\), so this entire sector contributes \(O_{\mathrm{eff}}(Y^{-5/4})\).

It remains to average the good-sector approximation over the largest selected factor. Assign the value \(0\) to the empty product and the value \((\log J)/X\) to a nonempty product. For \(R\le m\le A\), \eqref{eq:partial-product} gives
\[
\mathbb P(J\le m)
=
\frac{Z_m}{Z_A}
=
\frac{\ell_m}{\ell_A}
+
O_{\mathrm{eff}}(Y^{-3/2}).
\]
Uniformly in this range, the indexed-prime bounds give \(\ell_m=\log m+O_{\mathrm{eff}}(Y)\) and \(\ell_A=X+Y+O_{\mathrm{eff}}(1)\), whence \(\sup_{R\le m\le A}|\ell_m/\ell_A-(\log m)/X|\ll_{\mathrm{eff}}Y/X\). Below \(R\), both the uniform mass \((\log R)/X\) and the empty-product mass \(Z_A^{-1}\) are \(O_{\mathrm{eff}}(Y/X)\), and the mesh of the points \((\log m)/X\) is smaller still. Thus the resulting distribution on \([0,1]\) has Kolmogorov distance \(O_{\mathrm{eff}}(Y^{-3/2})\) from the uniform law.

For fixed \(T=t\), the branchwise integrand defining \(h_\delta(u)\) increases up to a single cutoff and then drops to zero. Its total variation is at most \(2/\delta\), and averaging in \(t\) preserves the bound \(\operatorname{TV}(h_\delta)\le2/\delta\). The one-dimensional Stieltjes inequality consequently gives
\[
\left|
\mathbb E h_\delta\!\left(\frac{\log J}{X}\right)
-
\int_0^1h_\delta(u)\,du
\right|
=
O_{\mathrm{eff}}(Y^{-5/4}),
\]
where the argument of \(h_\delta\) is understood to be \(0\) on the empty branch.

The computation in the proof of Theorem~\ref{thm:sharp-construction} identifies
\[
\int_0^1 h_\delta(u)\,du=c_\delta,
\qquad
0\le e^{-\gamma}-c_\delta\le\delta.
\]
Combining the good-sector smoothing estimate, the scale perturbation, the small-sector estimate, and the averaging estimate, with \(\delta=Y^{-1/4}\), yields
\[
\mathbb E\left[
\frac{
\mathbf1_{\{SJ\le A^{1-\delta}\}}
}{
1-\log S/X+3Y/X
}
\right]
\ge
e^{-\gamma}
-
Y^{-1/4}
-
O_{\mathrm{eff}}(Y^{-1/2}).
\]
Here and below this payoff is defined to be zero off the indicated event, so the denominator is evaluated only on its support. On the empty branch \(J=0\), the event holds and the displayed payoff equals \((1+3Y/X)^{-1}\).

We return to the deterministic count for \(\mathcal C_n\), setting \(\kappa=1+r/R\) and \(D_s=\log(A/s)+3Y\). By \eqref{eq:partial-product} and the indexed-prime estimates \(\ell_A=X+Y+O_{\mathrm{eff}}(1)\) and \(\ell_R=Y+\tfrac12\log Y+O_{\mathrm{eff}}(1)\),
\[
Z_A
=
\frac XY
\left(
1+
O_{\mathrm{eff}}\left(\frac{\log Y}{Y}\right)
\right),
\qquad
\kappa=1+O_{\mathrm{eff}}(Y^{-1/2}).
\]

The reciprocal probability space contains the empty product, whereas the family \(\mathcal C_n\) uses only nonempty multipliers. The exact conversion is therefore
\[
\begin{aligned}
\frac X\kappa
\sum_{\substack{
s\ne1\\
s\,i(s)\le A^{1-\delta}
}}
\frac1{sD_s}
&=
\frac{Z_A}{\kappa}
\mathbb E\left[
\frac{
\mathbf1_{\{SJ\le A^{1-\delta}\}}
}{
1-\log S/X+3Y/X
}
\right]\\
&\quad-
\frac{X}{\kappa(X+3Y)}.
\end{aligned}
\]
The final term therefore subtracts the empty-product contribution exactly once.

The endpoints \(J_s\) give the deterministic counting inequality
\[
\frac{|\mathcal C_n|}{A}
\ge
\frac X\kappa
\sum_{\substack{
s\ne1\\
s\,i(s)\le A^{1-\delta}
}}
\frac1{sD_s}
-
\frac1A
\sum_{\substack{
s\ne1\\
s\,i(s)\le A^{1-\delta}
}}
(i(s)+1).
\]
For the deleted lower indices, retention gives
\[
\frac1A
\sum_{\substack{
s\ne1\\
s\,i(s)\le A^{1-\delta}
}}
i(s)
\le
A^{-\delta}Z_A.
\]
The number of retained multipliers divided by \(A\) is at most \(A^{-\delta}/R\). Since \(A^{-\delta}=\exp(-X/Y^{1/4})\), both errors are smaller than every fixed negative power of \(Y\) after an effective threshold. The empty-product term is \(O(1)\), hence has normalized size \(O(Y/X)\) relative to the scale \(X/Y\). Proposition~\ref{prop:construction-validity} supplies the product-size, coprimality, compositeness, endpoint, and injectivity properties for precisely these same retained multipliers.

Combining the preceding estimates gives
\[
\frac{|\mathcal C_n|}{A}
\ge
\left[
e^{-\gamma}
-
Y^{-1/4}
-
O_{\mathrm{eff}}(Y^{-1/2})
-
O_{\mathrm{eff}}\left(\frac{\log Y}{Y}\right)
\right]
\frac XY.
\]
The additional terms \(Y^{21/4}/X\), \(\exp(-X/Y^{1/4})\), and \(Y/X\) are all effectively subordinate to \(Y^{-1/4}\). Hence there exist absolute effectively computable constants \(C_0,A_0>0\) such that
\[
\frac{|\mathcal C_n|}{A(n)}
\ge
\left(
e^{-\gamma}
-
C_0(\log\log A(n))^{-1/4}
\right)
\frac{\log A(n)}{\log\log A(n)}
\]
whenever \(A(n)\ge A_0\). The constants \(C_0\) and \(A_0\) can be obtained through a finite explicit procedure from the explicit constants and thresholds occurring in the proof; no numerical values for them are asserted.
\end{proof}

For the full ratios, the corresponding scale and extremal constants are classical.

\begin{proposition}\label{prop:classical-optimality}
One has
\[
\begin{aligned}
\liminf_{n\to\infty}
\frac{\phi(n)}{\pi(n)}
\frac{\log\log n}{\log n}
&=e^{-\gamma},\\
\liminf_{n\to\infty}
\frac{B(n)}{A(n)}
\frac{\log\log n}{\log n}
&=e^{-\gamma},\\
\limsup_{n\to\infty}
\frac1{\log n}\frac{\phi(n)}{\pi(n)}
&=1,\\
\limsup_{n\to\infty}
\frac1{\log n}\frac{B(n)}{A(n)}
&=1.
\end{aligned}
\]
Moreover, uniformly as \(n\to\infty\),
\[
\frac{B(n)}{A(n)}
=
\frac{\phi(n)}{\pi(n)}
\left(
1+
O\left(
\frac{\log\log n}{\log n}
\right)
\right).
\]
\end{proposition}

\begin{proof}
The prime number theorem gives \(\pi(n)\sim n/\log n\) \cite[(27.2.3)]{DLMF}. Since \(2^{\omega(n)}\le n\), we have \(\omega(n)=O(\log n)\), and hence \(\omega(n)/\pi(n)=O((\log n)^2/n)\). Therefore
\begin{equation}\label{eq:A-asymptotic}
A(n)
=
\pi(n)
\left(
1+
O\left(\frac{(\log n)^2}{n}\right)
\right)
\sim
\frac n{\log n}.
\end{equation}

Landau's minimal-order theorem states
\[
\liminf_{n\to\infty}
\frac{\phi(n)\log\log n}{n}
=
e^{-\gamma}
\]
\cite[Theorem 3.4.2]{LagariasEulerConstant}.
In particular, \(n/\phi(n)=O(\log\log n)\). Together with \eqref{eq:A-asymptotic}, this gives \(A(n)/\phi(n)=O(\log\log n/\log n)\), and hence, since \(B(n)=\phi(n)-A(n)\), one has \(B(n)=\phi(n)(1+O(\log\log n/\log n))\). Combining this with \eqref{eq:A-asymptotic} gives
\begin{equation}\label{eq:BA-vs-phi-pi}
\frac{B(n)}{A(n)}
=
\frac{\phi(n)}{\pi(n)}
\left(
1+
O\left(\frac{\log\log n}{\log n}\right)
\right).
\end{equation}

Now
\[
\frac{\phi(n)}{\pi(n)}
\frac{\log\log n}{\log n}
=
\left(
\frac{\phi(n)\log\log n}{n}
\right)
\left(
\frac{n}{\pi(n)\log n}
\right),
\]
and the second factor tends to \(1\). Landau's theorem therefore gives the first liminf assertion, and the second follows from \eqref{eq:BA-vs-phi-pi}.

For the upper envelope, \(\phi(n)\le n\) and the prime number theorem give
\[
\limsup_{n\to\infty}
\frac1{\log n}\frac{\phi(n)}{\pi(n)}
\le1.
\]
Along the primes \(p\), one has \(\phi(p)=p-1\), so the prime number theorem gives the reverse inequality. Hence the displayed limsup equals \(1\), and the corresponding assertion for \(B(n)/A(n)\) follows from \eqref{eq:BA-vs-phi-pi}.
\end{proof}

Equation \eqref{eq:A-asymptotic} also gives
\(\log A(n)/\log\log A(n)\sim\log n/\log\log n\).
Hence Theorem~\ref{thm:sharp-construction} may equivalently be read on the classical \(n\)-scale as
\[
\frac{|\mathcal C_n|}{A(n)}
\ge
\left(e^{-\gamma}-o(1)\right)
\frac{\log n}{\log\log n}.
\]
Consequently \(e^{-\gamma}\) is the largest possible leading constant in the following uniform comparison class. Suppose that, for all sufficiently large \(n\), a family \(\mathcal D_n\) consists of integers counted by \(B(n)\) and satisfies
\[
\frac{|\mathcal D_n|}{A(n)}
\ge
(c-o(1))
\frac{\log A(n)}{\log\log A(n)}.
\]
Since \(|\mathcal D_n|\le B(n)\), Proposition~\ref{prop:classical-optimality} together with the preceding asymptotic forces \(c\le e^{-\gamma}\). Thus the constant \(e^{-\gamma}\) and the corresponding extremal scale arise from the classical theorems of Landau and the prime number theorem. Theorem~\ref{thm:sharp-construction} realizes this classical ceiling constructively through the explicitly defined family \(\mathcal C_n\), and Theorem~\ref{thm:effective-construction} gives an effective rate for that same family.

\section{Exact Thresholds}\label{sec:thresholds}

Having established the sharp constructive lower bound, its quantitative refinement, and the classical comparison with the full ratios, we turn to a different question: the least eventual thresholds for prescribed strict dominance inequalities. The argument combines a general relation between the two threshold sequences with primorial extremality, support reduction, and certified finite computation.

\subsection{Structural and Analytic Reductions}\label{subsec:threshold-reductions}

For each integer \(k\ge1\), let \(N_k\) denote the least integer \(N\) such that \(\phi(n)>k\pi(n)\) for every \(n\ge N\), and let \(M_k\) denote the least integer \(M\) such that \(B(n)>kA(n)\) for every \(n\ge M\). Since these inequalities are strict, equality is counted as a failure.

\begin{proposition}\label{prop:threshold-relation}
For every integer \(k\ge1\), one has \(M_k\le N_{k+1}\).
\end{proposition}

\begin{proof}
If \(n\ge N_{k+1}\), then \(\phi(n)>(k+1)\pi(n)\). Since \(A(n)=\pi(n)-\omega(n)\le\pi(n)\), it follows that \(\phi(n)>(k+1)A(n)\), and therefore \(B(n)=\phi(n)-A(n)>kA(n)\). Thus every \(n\ge N_{k+1}\) satisfies the defining inequality for \(M_k\), so \(M_k\le N_{k+1}\).
\end{proof}

The threshold determination rests on two elementary reductions. The first controls \(n/\phi(n)\) by the extremal prime support of a given size; the second reduces a bounded range with fixed support size to finitely many radicals and admissible multipliers.

\begin{lemma}\label{lem:primorial-extremality}
Let \(n\ge2\) and \(r=\omega(n)\). Then
\[
\frac{n}{\phi(n)}
\le
\prod_{i=1}^{r}\frac{p_i}{p_i-1}.
\]
In particular, if \(j\ge1\) and \(p_j^\#\le n<p_{j+1}^\#\), then \(r\le j\) and
\[
\frac{n}{\phi(n)}
\le
\prod_{i=1}^{j}\frac{p_i}{p_i-1}.
\]
The latter bound is sharp on this primorial interval, with equality at \(n=p_j^\#\).
\end{lemma}

\begin{proof}
Let \(\ell_1<\cdots<\ell_r\) be the distinct prime divisors of \(n\). Since \(\ell_i\ge p_i\) and \(x/(x-1)\) is decreasing for \(x>1\),
\[
\frac{n}{\phi(n)}
=
\prod_{i=1}^{r}\frac{\ell_i}{\ell_i-1}
\le
\prod_{i=1}^{r}\frac{p_i}{p_i-1}.
\]
If \(p_j^\#\le n<p_{j+1}^\#\), then \(r\le j\), since otherwise \(n\ge p_r^\#\ge p_{j+1}^\#\), a contradiction. The second bound follows, and equality holds at \(n=p_j^\#\).
\end{proof}

\begin{lemma}\label{lem:support-reduction}
Let \(C>2\) and \(2\le n<C\). Put \(r=\omega(n)\) and
\(d=\operatorname{rad}(n)=\prod_{p\mid n}p\). Then \(d\ge p_r^\#\) and
\(n=md\) for some \(m\in\mathbb N\) satisfying \(m<C/p_r^\#\), and every
prime divisor of \(m\) divides \(d\).

If \(r\ge2\) and \(q\) is the largest prime divisor of \(d\), then
\(q\,p_{r-1}^\#\le d<C\), and hence \(q<C/p_{r-1}^\#\). Moreover,
\[
\frac{\phi(n)}{n}
=
\frac{\phi(d)}{d}
=
\prod_{p\mid d}\left(1-\frac1p\right).
\]
Consequently, for fixed \(C\) and \(r\), the integers \(n<C\) with
\(\omega(n)=r\) reduce to finitely many prime supports and finitely many
admissible multipliers \(m=n/d\).
\end{lemma}

\begin{proof}
Since \(d\) is the product of the \(r\) distinct prime divisors of \(n\), one
has \(d\ge p_r^\#\). Also \(d\mid n\), so \(n=md\) for some
\(m\in\mathbb N\). Every prime divisor of \(m\) already divides \(d\), and
\(m=n/d<C/d\le C/p_r^\#\).

If \(r\ge2\) and \(q\) is the largest prime divisor of \(d\), then the
remaining \(r-1\) prime divisors have product at least \(p_{r-1}^\#\). Hence
\(q\,p_{r-1}^\#\le d<C\), which gives the asserted bound on \(q\).

Finally, \(\phi(n)/n\) depends only on the prime support of \(n\), so
\[
\frac{\phi(n)}{n}
=
\prod_{p\mid n}\left(1-\frac1p\right)
=
\prod_{p\mid d}\left(1-\frac1p\right)
=
\frac{\phi(d)}{d}.
\]
The bounds on \(q\) and \(m\) make both the possible supports and the
possible multipliers finite.
\end{proof}

\subsection{Threshold Determination}\label{subsec:exact-thresholds}

\begin{theorem}\label{thm:exact-thresholds}
The exact eventual thresholds for the levels considered here are as follows:
\[
\begin{array}{c@{\qquad}r@{\qquad}r}
\hline
k & \text{last failure of }\phi(n)>k\pi(n) & N_k \\
\hline
1 & 90 & 91 \\
2 & 90\,090 & 90\,091 \\
3 & 223\,092\,870 & 223\,092\,871 \\
4 & 601\,681\,470\,390 & 601\,681\,470\,391 \\
5 & 2\,434\,002\,108\,217\,680 & 2\,434\,002\,108\,217\,681 \\
6 & 32\,589\,158\,477\,190\,044\,730 & 32\,589\,158\,477\,190\,044\,731 \\
\hline
\end{array}
\]
Moreover,
\[
M_k=N_{k+1}
\qquad(1\le k\le5).
\]
\end{theorem}

\begin{proof}
Throughout the finite computations, we work with integer-valued defects. Since \(A(n)=\pi(n)-\omega(n)\) and \(B(n)=\phi(n)-A(n)\), the two strict inequalities are equivalent to
\[
\begin{aligned}
\phi(n)>k\pi(n)
&\iff
\phi(n)-k\pi(n)>0,\\
B(n)>kA(n)
&\iff
\phi(n)-(k+1)\bigl(\pi(n)-\omega(n)\bigr)>0.
\end{aligned}
\]
Thus the finite scans involve exact integer arithmetic; the numerical analytic comparisons used below to close finite ranges and tails are certified by outward-rounded interval arithmetic.

For the first three \(\phi/\pi\) levels, an exhaustive segmented sieve through
\(266\,483\,380\), together with the indicated analytic tails, gives
\[
\begin{array}{l@{\qquad}r@{\qquad}l}
\hline
\text{inequality} & \text{last failure in the exact scan} & \text{certified tail} \\
\hline
\phi(n)>\pi(n) & 90 & n\ge296 \\
\phi(n)>2\pi(n) & 90\,090 & n\ge104\,065 \\
\phi(n)>3\pi(n) & 223\,092\,870 & n\ge266\,483\,381 \\
\hline
\end{array}
\]
The explicit estimates used to certify these tails are those of
Rosser--Schoenfeld, Büthe, and Axler
\cite[Theorem 15, p.~72]{RosserSchoenfeld1962}
\cite[Theorem 2(f)]{Buthe2018}
\cite[Theorem 2]{Axler2024}.
Hence \(N_1=91\), \(N_2=90\,091\), and \(N_3=223\,092\,871\). The same exact scan verifies that \(90\,090=N_2-1\) is a failure of \(B(n)>A(n)\), and that \(223\,092\,870=N_3-1\) is a failure of \(B(n)>2A(n)\). Therefore \(M_1\ge N_2\) and \(M_2\ge N_3\). Proposition~\ref{prop:threshold-relation} gives the reverse inequalities, so \(M_1=N_2=90\,091\) and \(M_2=N_3=223\,092\,871\).

For level \(4\), the same explicit estimates also certify
\(\phi(n)>4\pi(n)\) for every \(n\ge656\,835\,307\,560\).
This is the first case requiring the support reduction.
Put \(L_0=601\,681\,470\,390\) and \(C=656\,835\,307\,560\).
The certified boundary computation shows that \(L_0\) is a failure of both
\(\phi(n)>4\pi(n)\) and \(B(n)>3A(n)\), while the preceding analytic
estimate excludes all failures for \(n\ge C\). Thus it remains only to
exclude failures in \(L_0<n<C\).

Since \(p_{12}^\#=7\,420\,738\,134\,810>C\), one has
\(\omega(n)\le11\) throughout this interval. For \(\omega(n)\le10\),
Lemma~\ref{lem:primorial-extremality}, together with Axler's explicit
upper bound for \(\pi(n)\) \cite[Theorem 2]{Axler2024}, eliminates the
entire stratum. Hence only \(\omega(n)=11\) remains.

For this last stratum, Lemma~\ref{lem:support-reduction}, with
\(d=\operatorname{rad}(n)\), gives \(n/d\in\{1,2,3\}\). If \(q\) is the
largest prime divisor of \(d\), the same lemma gives
\(q\,p_{10}^\#<C\), hence \(q\le101\). Thus only finitely many prime
supports and admissible multipliers remain, and for each such support
\(\phi(n)/n=\phi(d)/d\) exactly. The certified exhaustive enumeration of these possibilities, combined
with the same explicit prime-counting bound, finds no failure of
\(\phi(n)>4\pi(n)\) in the open interval \((L_0,C)\). Therefore
\(N_4=L_0+1\). Since \(L_0=N_4-1\) is also a failure of
\(B(n)>3A(n)\), Proposition~\ref{prop:threshold-relation} gives
\(M_3=N_4\). Hence,
\[
N_4=M_3=601\,681\,470\,391.
\]

We next treat level \(5\). Put
\(L_*=8p_{13}^\#=2\,434\,002\,108\,217\,680\).
The exact values are \(\omega(L_*)=13\) and
\(\phi(L_*)=353\,158\,299\,648\,000\).
Dusart's lower bound
\(x/(\log x-1)\le\pi(x)\) for \(x\ge5393\)
\cite[Theorem 6.9, (6.6)]{Dusart2010},
together with certified outward-rounded interval computations, gives
\[
\begin{aligned}
\phi(L_*)
&<
5\frac{L_*}{\log L_*-1}
\le
5\pi(L_*),\\
\phi(L_*)
&<
5\left(
\frac{L_*}{\log L_*-1}-13
\right)
\le
5\bigl(\pi(L_*)-13\bigr).
\end{aligned}
\]
Thus \(L_*\) is a failure of the inequality defining \(N_5\); since
\(\omega(L_*)=13\), it is also a failure of \(B(n)>4A(n)\).

To exclude later failures of the first inequality, put
\(C=2.6\times10^{15}\).
On \(L_*<n<C\), the strata with \(\omega(n)\le12\) are eliminated by
the certified analytic comparisons obtained from
Lemma~\ref{lem:primorial-extremality} and Axler's explicit upper bound
for \(\pi(n)\) \cite[Theorem 2]{Axler2024}. Since \(p_{14}^\#>C\),
only the stratum \(\omega(n)=13\) remains.

For this stratum, Lemma~\ref{lem:support-reduction}, with
\(d=\operatorname{rad}(n)\) and \(n=md\), gives \(1\le m\le8\).
If \(q\) denotes the largest prime divisor of \(d\), the same lemma gives
\(q\,p_{12}^\#<C\), hence \(q<351\).
Thus the terminal range reduces to finitely many prime supports and
admissible multipliers. Using the exact support value
\(\phi(n)/n=\phi(d)/d\), the certified exhaustive enumeration finds no
failure of \(\phi(n)>5\pi(n)\) in the open interval \((L_*,C)\).

The remaining tail is covered by the same primorial mechanism.
On \(C\le n<p_{14}^\#\), the \(13\)-prime primorial bound has a strictly
positive certified margin, and the corresponding margins remain positive
across the transitions through \(p_{14}^\#\), \(p_{15}^\#\), and
\(p_{16}^\#\). From \(p_{16}^\#\) onward, the Rosser--Schoenfeld totient
estimate combined with Axler's prime-counting bound gives an increasing
lower quotient whose value at the initial point already exceeds \(5\).
Consequently \(\phi(n)>5\pi(n)\) for every \(n>L_*\), so
\(N_5=L_*+1\). Since \(L_*=N_5-1\) is also a failure of
\(B(n)>4A(n)\), Proposition~\ref{prop:threshold-relation} gives
\(M_4=N_5\). Hence
\[
N_5=M_4=2\,434\,002\,108\,217\,681.
\]

Finally, consider level \(6\). Put
\(P_{16}=p_{16}^\#=32\,589\,158\,477\,190\,044\,730\).
Certified outward-rounded interval computations, together with Dusart's
lower bound \(x/(\log x-1)\le\pi(x)\) for \(x\ge5393\)
\cite[Theorem 6.9, (6.6)]{Dusart2010}, give
\[
\begin{aligned}
\phi(P_{16})
&<
\frac{6P_{16}}{\log P_{16}-1}
\le
6\pi(P_{16}),\\
\phi(P_{16})
&<
6\left(
\frac{P_{16}}{\log P_{16}-1}-16
\right)
\le
6\bigl(\pi(P_{16})-16\bigr).
\end{aligned}
\]
Thus \(P_{16}\) is a failure of the inequality defining \(N_6\), and,
since \(\omega(P_{16})=16\), it is also a failure of \(B(n)>5A(n)\).

It remains to exclude later failures of the first inequality. Put
\(C=4\times10^{19}\). Since \(p_{17}^\#>C\), one has
\(\omega(n)\le16\) for every \(n<C\). For the strata with
\(\omega(n)\le15\), the certified analytic comparison obtained from
Lemma~\ref{lem:primorial-extremality} and Axler's explicit upper bound
for \(\pi(n)\) \cite[Theorem 2]{Axler2024} shows that
\(\phi(n)>6\pi(n)\) throughout \(P_{16}<n<C\) when
\(\omega(n)\le15\). Hence only the stratum \(\omega(n)=16\) remains.

For this stratum, Lemma~\ref{lem:support-reduction}, with
\(d=\operatorname{rad}(n)\), gives \(1\le n/d<C/P_{16}<2\), so
\(n=d\) is squarefree. If \(q\) is the largest prime divisor of \(n\),
the same lemma gives \(q\,p_{15}^\#<C\), hence \(q<66\). The resulting
finite support enumeration leaves, besides the boundary \(P_{16}\),
exactly two candidates below \(C\): those obtained from \(P_{16}\) by
replacing \(53\) by \(59\) and by \(61\), respectively. Both have
strictly positive certified support-specific margins for
\(\phi(n)>6\pi(n)\). Thus no failure occurs in the open interval
\((P_{16},C)\).

For \(C\le n<p_{17}^\#\), the \(16\)-prime primorial bound has a
strictly positive certified margin above level \(6\). From
\(p_{17}^\#\) onward, the Rosser--Schoenfeld totient estimate combined
with Axler's prime-counting bound gives an increasing lower quotient
whose value at \(p_{17}^\#\) already exceeds \(6\). Consequently
\(\phi(n)>6\pi(n)\) for every \(n>P_{16}\), so \(N_6=P_{16}+1\).
Since \(P_{16}=N_6-1\) is also a failure of \(B(n)>5A(n)\), one has
\(M_5\ge N_6\); Proposition~\ref{prop:threshold-relation} gives the
reverse inequality \(M_5\le N_6\). Hence,
\[
N_6=M_5=32\,589\,158\,477\,190\,044\,731.
\]

All finite scans and support enumerations used above are performed with exact integer or rational arithmetic. The reductions in Lemmas~\ref{lem:primorial-extremality} and~\ref{lem:support-reduction} make the bounded searches exhaustive, while the analytic comparisons needed for boundary strictness and for closing the remaining tails are certified by outward-rounded interval arithmetic. Available computational source files and certificate data are collected in the accompanying reproducibility material.
\end{proof}

For every positive integer \(k\), the exact decompositions give
\begin{equation}\label{eq:threshold-defect}
B(n)-kA(n)
=
\phi(n)-(k+1)\pi(n)
+
(k+1)\omega(n).
\end{equation}
The defect identity \eqref{eq:threshold-defect} shows that the two threshold problems differ by the nonnegative correction \((k+1)\omega(n)\). Nevertheless, Theorem~\ref{thm:exact-thresholds} gives \(M_k=N_{k+1}\) for \(1\le k\le5\). This repeated equality motivates the following conjecture.

\begin{conjecture}\label{conj:threshold-equality}
For every integer \(k\ge1\), $M_k=N_{k+1}$.
\end{conjecture}

\end{document}